\documentclass[11pt]{amsart}
\usepackage{amssymb, amsfonts, amsmath, amsthm}
\usepackage{enumerate}
\usepackage[mathscr]{eucal}
\usepackage{mathrsfs}
\usepackage{pgf,tikz}
 \usetikzlibrary{arrows}
\usepackage{url}

\newtheorem{theorem}{Theorem}

\newtheorem{lemma}[theorem]{Lemma}
\newtheorem{proposition}[theorem]{Proposition}
\theoremstyle{definition}

\theoremstyle{definition}
\newtheorem{definition}[theorem]{Definition}
\newtheorem{example}[theorem]{Example}
\newtheorem{remark}[theorem]{Remark}
\newtheorem{algorithm}[theorem]{Algorithm}

\begin{document}

\title{Almost symmetric  numerical semigroups with given Frobenius number and type}

\author{M.B. Branco}
\address{Universidade de \'Evora\newline 
\indent Departamento de Matem\'atica\newline 
\indent 7000 \'Evora, Portugal}
\email{mbb@uevora.pt}

\author{I. Ojeda}
\address{Universidad de Extremadura\newline 
\indent Departamento de Matem\'aticas\newline 
\indent E-06071 Badajoz, Spain}
\email{ojedamc@unex.es}

\author{J.C. Rosales}
\address{Universidad de Granada\newline
\indent Departamento de \'Algebra\newline 
\indent E-18071 Granada, Spain}
\email{jrosales@ugr.es}

\thanks{The first author is supported by the project FCT PTDC/MAT/73544/2006). The second author was partially supported by the research groups FQM-024 (Junta de Extremadura/FEDER funds) and by the project MTM2015-65764-C3-1-P (MINECO/FEDER, UE). 
The third author was partially supported by the research groups FQM-343 FQM-5849 (Junta de Andalucia/Feder) and by the project  MTM2014-55367-P (MINECO/FEDER, UE. 2010 Mathematics Subject Classification: 20M14, 11D07.}

\date{\today}
\subjclass[2010]{20M14, 20M25.}
\keywords{\em Almost symmetric numerical semigroup, irreducible numerical semigroup, genus, Frobenius number, type.}

\begin{abstract}
We give two algorithmic procedures to compute the whole set of almost symmetric numerical semigroups with fixed Frobenius number and type, and the whole set of almost symmetric numerical semigroups with fixed Frobenius number. Our algorithms allow to compute the whole set of almost symmetric numerical semigroups with fixed Frobenius number with similar or even higher efficiency that the known ones. They have been implemented in the GAP (\cite{GAP}) package \texttt{NumericalSgps} (\cite{numericalsgps}).
\end{abstract}

\maketitle
\pagestyle{myheadings}
\markboth{{M. B.} Branco \and I. Ojeda \and {J. C.} Rosales}{Almost symmetric numerical semigroups}

\section{Introduction}

Let $\mathbb{Z}$ and $\mathbb{N}$ denote the set of integers and nonnegative integers, respectively. A \textbf{numerical semigroup} is a subset $S$ of $\mathbb{N}$ that is closed under addition, contains the zero element and has finite complement in $\mathbb{N}$, that is, $\mathbb{N} \setminus S$ is a finite set.

Given a nonempty subset $A$ of $\mathbb{N}$, we write $\langle A \rangle$ for the submonoid of $\mathbb{N}$ generated by $A$, that is, \[\langle A \rangle := \Big\{ \sum_{finite} n a \mid n \in \mathbb{N}, a \in A\Big\}\]

It is well known that $\left\langle A\right\rangle$ is a numerical semigroup if and only if $\gcd\left(A\right)=1$ (see \cite[Lemma 2.1]{libro}) .  If $S$ is a numerical semigroup and $S=\left\langle A\right\rangle$ then  we say that $A$ is a \textbf{system of generators} of $S$.  Furthermore if $S\neq \left\langle A'\right\rangle$ for every $A' \varsubsetneq A$, then we say that $A$ is a \textbf{minimal system of generators} of $S$. It is also known that every numerical semigroup admits a unique finite minimal system of generators (\cite[Lemma 2.7]{libro}). Hence, if $S$ is a numerical semigroup, we write $\mathrm{msg}(S)$ for the  minimal system of generators of $S$.

Here and in the following $S$ will denote a numerical semigroup.

The \textbf{Frobenius number} of $S$ is the largest integer not belonging to $S$ and it is denoted by $\mathrm{F}(S)$. A numerical semigroup is irreducible if it cannot be expressed as the intersection of two numerical semigroups properly containing it. A numerical semigroup is said to be \textbf{symmetric} (\textbf{pseudo-symmetric}, resp.) if it is irreducible and its Frobenius number is odd (even, resp).

An integer $x\in \mathbb{Z}\setminus S$ is called a \textbf{pseudo-Frobenius number} if $x+S\setminus\{0\}\subseteq S$ (see \cite{JPAA}). The set of pseudo-Frobenius number of $S$ is denoted by $\mathrm{PF}(S)$. The type of the semigroup, denoted by $\mathrm{t}(S)$, is the cardinality of $\mathrm{PF}(S)$ and it was introduced in  \cite{froberg}.

As a consequence of \cite[Corollary 4.11 and 4.16 ]{libro}, we have the following result.

\begin{lemma}\label{2}
If $S$ is a numerical semigroup, then:
\begin{enumerate}[ (a)]
  \item $S$ is symmetric if and only if $\mathrm{PF}(S)=\big\{\mathrm{F}(S)\big\}$.
  \item  $S$ is pseudo symmetric if only if $\mathrm{PF}(S)=\big\{\mathrm{F}(S),\frac{\mathrm{F}(S)}{2}\big\}$.
\end{enumerate}
\end{lemma}

Observe that if $S$ is an irreducible numerical semigroups then $\mathrm{t}(S)\in\{1,2\}$ and we have that $\mathrm{t}(S)=1$ if and only if $\mathrm{F}(S)$ is odd, equivalently if $S$ is symmetric.

The elements in $\mathbb{N}\setminus S$ are called the \textbf{gaps} of $S$ and the cardinality of the set of gaps of $S$, $\mathrm{g}(S)$, is known as the \textbf{genus} of $S$. Following the terminology of \cite{jager} (see also \cite{barucci-froberg}) the elements in $$\mathrm{N}(S) := \{x \in \mathbb{N} \setminus S \mid \mathrm{F}(S)-x\in S\}$$ are called \textbf{gaps of the first type} and the remaining gaps are called \textbf{gaps of the second type}. Set of gaps of the second type is denoted $\mathrm{L}(S)$. By \cite[Proposition 3.4]{libro}, one has that if $S$ is symmetric (pseudo-symmetric, resp.) then $\mathrm{L}(S) = \varnothing$ ($\mathrm{L}(S) = \{\mathrm{F}(S)/2 \}$, resp.).

\begin{definition}
A numerical semigroup $S$ is \textbf{almost symmetric} (AS-symme\-tric, for short) when $\mathrm{L}(S) \subseteq \mathrm{PF}(S)$. 
\end{definition}

Notice that symmetric and pseudo-symmetric semigroups are almost symmetric. 

The concept of AS-semigroup was introduced in \cite {barucci-froberg} where the next is proven. 

\begin{proposition}\label{1}
Let $S$ be a numerical semigroup. The following conditions are equivalent:
\begin{enumerate}[ (a)]
  \item $S$ is AS-semigroup;
  \item $\mathrm{PF}(S)=\mathrm{L}(S)\cup\{\mathrm{F}(S)\}$;
  \item $\mathrm{g}(S)=\frac{\mathrm{F}(S)+\mathrm{t}(S)}{2}$.
\end{enumerate}
\end{proposition}

It is well known (see \cite[Proposition 2.2]{nari}) that if $S$ is a numerical semigroup then \begin{equation}\label{ecu3} \mathrm{g}(S) \geq \frac{\mathrm{F}(S)+\mathrm{t}(S)}{2}.\end{equation} Therefore, we deduce that AS-semigroups are those numerical semigroups with the least possible genus in the set of numerical semigroups with fixed Frobenius number and type (observe that AS-semigroups are maximal in this set).

The main aim of this paper is to give algorithmic methods to compute the whole set of AS-semigroups with fixed Frobenius number and given type. Thus, by Proposition \ref{1}, we will solve the optimization of problem of the computation of all numerical semigroups with fixed Frobenius number and type, and minimal genus. 

An interesting application our computations is related to the theory of $S-$graded minimal free resolutions (see Remark \ref{Remark final}).

For our main purpose, we will give two algorithms. The first one will start from the almost symmetric numerical semigroups with Frobenius number $F$ and smallest possible type (that is, type equals to $1$) and the second will start from the almost symmetric numerical semigroup with Frobenius number $F$ and biggest possible type  (that is, type equals to $F$). For that reasons, we have called them the ascending and the descending algorithms, respectively.

As an inmediate consequence, our algorithms  allow to compute set of AS-semigroups with fixed Frobenius number. The asceding algorithm has a running time similar to the algorithm given in \cite{CA}. Its main virtue is that it enables to perform the computation with separate processes for each feasible type or even truncate the computation in a range of types. The descending algorithm wins without forgiveness for the previous ones for high Frobenius numbers . A discussion table of timing is included in Section \ref{sect time}.

The referee pointed out an alternative algorithm to compute the set of AS-semigroups with fixed Frobenius number and given type that is based in the algorithm to compute the set of numerical semigroups with given pseudo-Frobenius numbers introduced in \cite{DGSRP}. We briefly discuss this alternative algorithm in Section \ref{sect time}.

We end this paper with a computational evidence which establishes an apparently relation between the number of AS-semigroups with fixed Frobenius number and given type and the Bras-Amoros' conjecture on the number of numerical semigroups of a given genus (see Remark \ref{rem_fin}).

\section{Irreducible numerical semigroups and AS-semigroups}\label{S2}

Throughout this section, $F$ and $t$ are integers and $S$ denotes a numerical semigroup. Our goal now is to describe the conditions that $F$ and $t$ must fulfill for $S$ to be an AS-semigroup with  Frobenius number $F$ and type $t$. First we need to recall some useful results.

\begin{proposition}\cite[Theorem 3]{CA}.\label{3}
$S$ is AS-semigroup if and only if there exists an irreducible numerical $S'$ and $A \subseteq \mathrm{msg}(S')$ such that
\begin{enumerate}[ (a)]
\item $S = S' \setminus A$.
\item $\mathrm{F}(S')/2 < x < F(S')$, for every $x \in A$.
\item $x+y-\mathrm{F}(S') \not\in S$, for every $x,y \in A$.  
\end{enumerate}
In this case, $\mathrm{F}(S) = \mathrm{F}(S')$ and $\mathrm{t}(S)=2 \#A + \mathrm{t}(S')$, where $\#$ means cardinality.
\end{proposition}

Notice that if $A \neq \varnothing$, then $S' \setminus A$ is neither symmetric nor pseudo-symmetric.

Let us denote by $C(F)$ the following numerical semigroup:
\begin{itemize}
\item $\{0, \frac{F+1}2, \frac{F+1}2+1, \frac{F+1}2+2, \ldots \} \setminus \{F\} $ if $F$ is odd; 
\item $\{0, \frac{F}2+1, \frac{F}2+2, \ldots \} \setminus \{F\} $ if $F$ is even. 
\end{itemize}
Clearly, $C(F)$ has Frobenius number equal to $F$. Moreover, by Lemma \ref{2}, $C(F)$ is irreducible.

\begin{proposition}\cite[Proposition 2.7]{forum}\label{5}
The semigroup $C(F)$ is the unique irreducible numerical semigroup
with Frobenius number $F$ and such that all its minimal generators are larger than $\frac{F}2$.
\end{proposition}

In the following result we present a necessary and sufficient condition for the existence of an AS-semigroup with given Frobenius number and type.

\begin{proposition}\label{6}
There exists an almost symmetric semigroup $S$ with $F(S) = F$
and $t(S) = t$ if if and only if $\mathrm{F}(S)+\mathrm{t}(S)$ is an even number less than or equal to $2\,  \mathrm{F}(S)$.
\end{proposition}

\begin{proof}
Set $F = \mathrm{F}(S)$ and $t = \mathrm{t}(S)$. If $S$ is an AS-semigroup, by Proposition \ref{1}, $F+t = 2 \mathrm{g}(S)$ is even. Besides, since  $g \leq F$, it follows that $F+t \leq 2 F$. Conversely, we distinguish to cases:   
\begin{itemize}
\item[] if $F$ is odd, we set $A=\left\{\frac{F+1}{2},\frac{F+1}{2}+1,\ldots,\frac{F+1}{2}+\frac{t-3}{2}\right\}$; 
\item[] if $F$ is even, we set $A=\left\{\frac{F}{2}+1,\ldots,\frac{F}{2}+\frac{t-2}{2}\right\}$.
\end{itemize}
Now, if $S' = C(F)$, then, by applying Proposition \ref{3}, we conclude  that $S' \setminus A$ is an AS-semigroup with Frobenius number $F$ and type $t$.
\end{proof}

Observe that, if we choose $ F $ and $ t $ as in Proposition \ref{6}, then we can use Proposition \ref{3} to produce an AS-semigroup with Frobenius number $ F $ and type $ t $, from $ C (S) $. Let us see a couple of examples.

\begin{example}\label{7a}
Set $F=11$ and $t = 5$. By Proposition \ref{6}, there exists at least one AS-semigroup with Frobenius $11$ and type $5$. In this case, the semigroup $C(11)$ is $\{0,6,7,8,9,10,12,13, \ldots \}$. So, by taking $A = \{6,7\}$, Proposition \ref{3} ensures that \[S = C(11) \setminus A = \{0,8,9,10,12,13, \ldots \}\] is an AS-semigroup with $\mathrm{F}(S)=11$ and $\mathrm{t}(S)=5$.
\end{example}

\begin{example}\label{7b}
Set $F=10$ and $t = 6$. By Proposition \ref{6}, there exists at least one AS-semigroup with Frobenius $10$ and type $6$. In this case, the semigroup $C(10)$ is  $\{0,6,7,8,9,11,12,13, \ldots \}$. So, by taking $A = \{6,7\}$, Proposition \ref{3} ensures that \[S = C(10) \setminus A = \{0,8,9,11,12,13, \ldots \}\] is an AS-semigroup with $\mathrm{F}(S)=10$ and $\mathrm{t}(S)=6$.
\end{example}

\section{The asceding algorithm}\label{S3}

Let $F$ and $t$ be positive integers such that $F+t$ is an even number lesser than or equal to $2\, F$. The objective in this section is to describe an algorithmic procedure to compute the set \[\mathscr A(F,t):=\left\{S \mid S \text{ is an AS-semigroup with } \mathrm{F}(S)= F \text{ and } \mathrm{t}(S)=t \right\}.\]
In order to reach the almost symmetric numerical semigroups with Frobenius number $F$ and type $t$, our algorithm will depart from the almost symmetric numerical semigroups with Frobenius number $F$ and type $1$, that is to say, from the irreducible numerical semigroups with Frobenius number $F$.

Before presenting the result that supports our algorithm, we need to recall one of the main results in  \cite{forum} which allows to compute the set $\mathscr I(F)$ of all irreducible numerical semigroups with Frobenius number $F$.

\begin{proposition}\cite[Proposition 2.9]{forum}\label{9}
If $S\in \mathscr I(F)$, then \[\big( S \setminus \{x\} \big) \cup \{F- x \} \in 
\mathscr I(F),\] for every $x \in\mathrm{msg}(S)$ such that 
$ \frac{F}{2} < x < F,\ 2 x- F \not\in S,\  3 x \neq 2\, F,\ 4 x\neq 3\, F$ and $F-x< \min\big(\mathrm{msg}(S)\big)$.
\end{proposition}

Let $\prec$ be the following binary relation on $\mathscr I(F)$: $S \prec S'$ if and only if there exists $S = S_0, S_1, \ldots, S_p = S'$ such that $S_{i+1}$ is obtained from $S_i$ by applying Proposition \ref{9}. In \cite{forum} it is shown that $\prec$ is an order relation (reflexive, transitive and antisymmetric) on $\mathscr  I(F)$ with first element equal to $C(F)$. In other words, the set $\mathscr  I(F)$  forms a rooted tree with root $C(F )$.

Therefore, one can compute all the irreducible numerical semigroup with Frobenius number $F$ starting from $C(F)$.

\begin{example}\label{10}
Let us compute $\mathscr I(11)$. Using Proposition \ref{9}, we compute the elements in $\mathscr I(11)$ that are immediately bigger than $$C(11)=\{0, 6,7,8,9,10,12,13 \ldots \} = \langle 6,7,8,9,10 \rangle,$$ namely $\langle 3,7 \rangle, \langle 4,6,9 \rangle$ and $\langle 5,7,8,9 \rangle$. By repeating the same construction with these three semigroups, we obtain that the first is maximal and the other two produce $\langle 2, 13 \rangle$ and $\langle 4, 5 \rangle$, respectively. As these two last are maximal, we are done.

The following figure summarizes the construction flow:
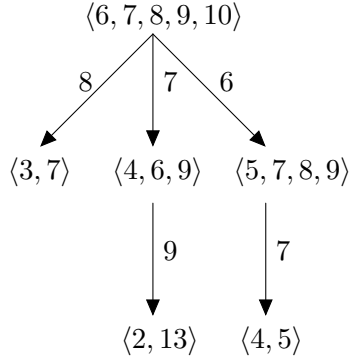
\begin{center}
\begin{tikzpicture}[line cap=round,line join=round,>=triangle 45,x=.75cm,y=.75cm]
\clip(2,-1) rectangle (9,6);
\draw [->] (5,5) -- (3,3);
\draw [->] (5,5) -- (5,3);
\draw [->] (5,5) -- (7,3);
\draw [->] (5,2) -- (5,0);
\draw [->] (7,2) -- (7,0);
\draw (3.6,5.75) node[anchor=north west] {$\langle 6,7,8,9,10\rangle$};
\draw (2.25,3) node[anchor=north west] {$ \langle 3,7 \rangle$};
\draw (4.1,3) node[anchor=north west] {$\langle 4,6,9\rangle$};
\draw (6.25,3) node[anchor=north west] {$\langle 5,7,8,9\rangle$};
\draw (4.25,0) node[anchor=north west] {$\langle 2,13\rangle$};
\draw (6.35,0) node[anchor=north west] {$\langle 4,5\rangle$};
\draw (3.5,4.5) node[anchor=north west] {$8$};
\draw (5,4.5) node[anchor=north west] {$7$};
\draw (6,4.5) node[anchor=north west] {$6$};
\draw (5,1.5) node[anchor=north west] {$9$};
\draw (7,1.5) node[anchor=north west] {$7$};
\end{tikzpicture}
\end{center}
where each edge was marked with the corresponding $x$.
\end{example}

Now, let us write $b(S)$ for the cardinality of the set \[\left\{ x\in\mathrm{msg}(S) \mid\frac{\mathrm{F}(S)}{2}<x<\mathrm{F}(S) \right\}\] and define  $\mathscr B(F,t) :=\left\{S' \in \mathscr I(F) \mid b(S')\geq  \lceil \frac{t}{2}\rceil-1 \right\}$, being $\lceil - \rceil$ the ceiling operator. Observe that 
$\mathscr B(F,t) \neq \varnothing$, because $C(F) \in \mathscr B(F,t)$.


\begin{theorem}\label{Th main}
With the notation above, $S$ is an almost symmetric semigroup with Frobenius number $F$ and type $t$ if and only if there exist $S' \in  \mathscr B(F,t)$ and a subset $A$ of $\mathrm{msg}(S')$ with cardinality $\lceil \frac{t}{2}\rceil-1$ such that $S = S' \setminus A,\ \mathrm{F}(S')/2 < x < F(S')$ and $x+y-\mathrm{F}(S') \not\in S'\setminus A$, for every $x,y \in A$.
\end{theorem}

\begin{proof}
Let $S$ be an AS-semigroup  with Frobenius number $F$ and type $t$. By Proposition \ref{3}, there exist $S' \in \mathscr I(F)$ and $A \subseteq \mathrm{msg}(S')$ such that $S = S' \setminus A, \mathrm{F}(S')/2 < x < F(S')$ and $x+y-\mathrm{F}(S') \not\in S'\setminus A$, for every $x,y \in A$. Moreover, since $t = \mathrm{t}(S) = 2 \# A + \mathrm{t}(S')$, we obtain that $\# A = \frac{t+\delta}2-1$ where $\delta$ is $0$ or $1$ depending on whether $F$ is even or odd. Now, $F+t$ is an even number, we conclude that $\# A =  \lceil \frac{t}{2}\rceil-1$. The converse follows directly from  Proposition \ref{3}.
\end{proof}

The following algorithm computes all AS-semigroups with Frobenius number $F$ and type $t$.

\begin{algorithm}\label{12}\mbox{}\par
\textsc{Input}: $F$ and $t$ two positive integers.\par
\textsc{Output}: $\mathscr A(F,t)$.
\begin{enumerate}[ 1)]
\item Set $\mathscr A(F,t) = \varnothing$.
\item If $t > F$ or $F+t$ is odd  then return $\mathscr A(F,t)$.
\item Compute the set $\mathscr I(F)$.
\item Compute  the set $\mathscr B(F,t) =\left\{S' \in \mathscr I(F) \mid b(S')\geq  \lceil\frac{t}{2}\rceil-1 \right\}$.
\item For each $S' \in \mathscr B(F,t)$ compute
\begin{enumerate}[ 4.1)]
\item $\mathcal A(S')=\Big\{A \subseteq \mathrm{msg}(S') \mid \#A=\lceil \frac{t}{2}\rceil-1,\ \frac{F}{2}<x<F$ and $x+y-F\not\in S' \setminus A$, for every $x,y \in A\Big\}$ 
\item Set $\mathscr D(S')=\left\{ S' \setminus A \mid A\in \mathcal A(S')\right\} $ 
\item Append $\mathscr D(S')$ to $\mathscr A(F,t)$.
\end{enumerate}
\item  Return $\mathscr A(F,t)$.
\end{enumerate}
\end{algorithm}

The correctness of this algorithm relies on Theorem \ref{Th main}.
The algorithm ends because the sets $\mathscr B(F,t)$ and $\mathcal A(S'),\ S' \in \mathscr B(F,t)$, are empty or finite. 

\begin{example}
Let us compute the set $\mathscr A(11,7)$. By Example \ref{10} we have that $\mathscr B(11,7)=\left\{\langle 6,7,8,9,10\rangle,
\langle 5,7,8,9\rangle\right\}$.
\begin{itemize}
\item If $S'=\langle 6,7,8,9,10\rangle$, then the subsets $A$ of $\left\{6,7,8,9,10\right\}$ with cardinality $3$ such that  $x+y-11\not\in S'\setminus A$, for every $x,y \in A$ are $\{6,7,8\}$ and $\{6,7,9\}$ 
\item If $S'=\langle 5,7,8,9\rangle$, then there exists no $A$ possible.
\end{itemize}
Hence, by applying Proposition \ref{3},
\begin{align*}
\mathscr A(11,7) & =\left\{ \langle 6,7,8,9,10\rangle\setminus \{6,7,8\}, \langle 6,7,8,9,10\rangle\setminus \{6,7,9\}\right\}\\ & =  \left\{\langle 9,10,12,13,14,15,16,17\rangle,\langle 8,10,12,13,14,15,17,19\rangle\right\}
\end{align*}
We can verify this computation by using the \texttt{NumericalSgps} package (see \cite{numericalsgps}). It suffices to change the first line in the code provided in \cite[Example 5]{CA} by 
\texttt{almostnsfromirreducibleAlt := function(s,t)} and set the variable \texttt{pow} as \texttt{Combinations(msg,CeilingOfRational(t/2)-1)}. 

\medskip
The following piece of code will do the rest:
\begin{verbatim}
   assemigroupswithfrobeniusnumberandtype:=function(f,t)
      return Union(Set(     
      IrreducibleNumericalSemigroupsWithFrobeniusNumber(f),
      s->almostnsfromirreducible(s,t)
      ));
   end;
\end{verbatim} 
Indeed, the command \texttt{assemigroupswithfrobeniusnumberandtype(11,7)} will give us all the almost symmetric semigroups of Frobenius number 11 and type 7.
\end{example}

\begin{remark}
Let $g$ be a positive integer. By Proposition \ref{1}, the set of all AS-semigroups with genus $g$ and type $t$ and the set of all AS-semigroups with Frobenius number $F$ and genus $g$ are $\mathscr A(2g-t,t)$ and $\mathscr A(F,2g-F)$, respectively.
Notice that,  by Proposition \ref{6}, there exists at least one AS-semigroup with genus $g$ and type $t$ for every pair $(g,t)$ of positive integers such that $t \leq g$. Analogously, there exists at least one AS-semigroup with Frobenius number $F$ and genus $g$ for every pair $(F,g)$ of positive integers  such that $g \leq F$.
\end{remark}

We, finally, observe that since $t \leq F$ we can compute the set of all AS-semigroup with fixed Frobenius number $F$, $\mathscr A(F)$,  that is nothing but the union of the sets $\mathscr A(F,2i-F)$ for every $i = \lceil \frac{F+1}2 \rceil, \ldots, F$.

\begin{verbatim}
assemigroupswithfrobeniusnumber := function(F)
    local irr, irrt, c;

    irr:=IrreducibleNumericalSemigroupsWithFrobeniusNumber(F);
    if IsOddInt(F) then 
      c:=1; 
    else 
      c:=2;
    fi;
    irrt:=Cartesian(irr,[c..Int(F/2)]);
    return Concatenation(irr,Concatenation(List(irrt,
     st->almostnsfromirreducibleAlt(st[1],2*st[2]+
     RemInt(F,2)))));
end;
\end{verbatim}

\section{The desceding algorithm}\label{S4}

Let $F$ be a positive integer and define $M(F)$ to be the numerical semigroup $\{0, F+1, F+2, \ldots \}$. Clearly, this semigroup has Frobenius number and type equal to $F$. Moreover, since $\mathbb{N} \setminus M(F) = \mathrm{PF}(M(F))= \{1, \ldots, F\}$, we have that $g = \frac{F+F}2$. So, by Proposition \ref{1}, $M(F)$ is a almost symmetric.

\begin{lemma}
The unique numerical semigroup with both Frobenius number and type equal to $F$ is $M(F)$.
\end{lemma}

\begin{proof}
Let $S$ be a numerical semigroup such that $F(S)=t(S)=F$. Since $g(S) \leq F$ and, by \eqref{ecu3}, $g(S) \geq \frac{F(S)+t(s)}2=F$. We conclude that $\mathbb N \setminus S = \{1, \ldots, F\}$, that is to say, $S = M(F)$.  
\end{proof}

In this section, we will depart from $M(F)$, i.e. the almost symmetric semigroup with Frobenius number and type equal to $F$, in order to ``descend'' to the set of almost symmetric semigroup with Frobenius number $F$ and type $t \leq F$.

We recall that the least positive integer in a numerical semigroup $S$ is called the \textbf{multiplicity} of $S$ and it is denoted $m(S)$. The following result relates the multiplicity and type of a numerical semigroup.

\begin{proposition}\label{15}\cite[Corollary 2.23]{libro}.
Let $S$ be a numerical semigroup. Then \[m(S) \geq t(S)+1\]
\end{proposition}

\begin{theorem}\label{16}
Let $F \geq 5$ and $t \leq F$ be a positive integer greater than $2$ such that $F+t$ is even. Then $S$ is an almost symmetric numerical semigroup with Frobenius number $F$ and type $t-2$ if and only if there exist an almost symmetric numerical semigroup $S'$ with Frobenius number $F$ and type $t$, $x \in [t-1,m(S')-1]$ such that 
\begin{enumerate}[ (a)]
\item $S = \{x\} \cup S'$.
\item $-x + \mathbb{N} \setminus S \subseteq \mathbb{Z} \setminus S$.
\item $\big(x + (\mathrm{PF}(S') \setminus \{x,F-x\}) \big) \subseteq S$.
\end{enumerate}
In this case, $\mathrm{PF}(S') = \mathrm{PF}(S) \cup \{x,F-x\}.$
\end{theorem}

\begin{proof}
Suppose that $S$ is almost symmetric with Frobenius number $F$ and type $t-2$. Set $x = m(S)$ and define $S'$ to be the numerical semigroup generated by $S \setminus \{x\}$. Notice that the conditions (a)-(c) are trivially satisfied by construction.

Let us prove that $S'$ is almost symmetric with Frobenius number $F$ and type $t$. 

Clearly, $F(S')$ is either $x$ or $F$. If $F(S') = x$, then $S = M(F)$ has type $F > t-2$ that is not possible by hypothesis. So $F(S') = F$. 

Now, we prove that the type of $S'$ is $t$. By definition, $g(S')=g(S)+1$. Thus, since \begin{equation}\label{ecu1} \frac{F+t(S')}2 \leq g(S') = g(S)+1=\frac{F+t-2}2+1 = \frac{F+t}2,\end{equation} we obtain that $t(S') \leq t$. Furthermore, $F-x \not\in S'$, otherwise $F=(F-x)+x \in S$, and $x=F-(F-x) \not\in S'$. Thus, $\{x,F-x\} \subseteq PF(S')$. Now, since $y + S \setminus \{0\} \subseteq S$, for every $y \in PF(S)$, we have that $y+S' \setminus \{0\} \subseteq S$, for every $y \in PF(S)$. If $y+s' = x$, for some $s' \in S' \setminus \{0\},$ then $s' = x-y \leq x = m(S) < m(S') \leq s'$, a contradiction.  Therefore $PF(S) \subseteq PF(S') \setminus \{x,F-x\},$ in particular, $t-2 = t(S) \leq t(S')-2,$ that is to say, $t \leq t(S')$. Thus, $t(S') = t$.

From \eqref{ecu1} it follows that $g(S') = \frac{F(S')+t(S')}2$; so, by Proposition \ref{1}, we conclude that $S'$ is almost symmetric.

Conversely, suppose that there exist an almost symmetric numerical semigroup $S'$ with Frobenius number $F$ and type $t$, and $x \in [t-1,m(S')-1]$ satisfying (a)-(c). Let $S = \{x\} \cup S'$. Given $y \in S$, then $x+y \in S$; otherwise, by condition (b), $y = -x+(x+y) \in \mathbb{Z} \setminus S$. Therefore, $S$ is a semigroup that is numerical because $\mathbb{N} \setminus S \subset \mathbb{N} \setminus S'$. Moreover, since $x \leq m(S') - 1 < F$, we have that $F(S) = F$. 

Let us see now that $t(S) = t-2$. On the one hand we have that $g(S) = g(S')-1$. Thus, \begin{equation}\label{ecu2} \frac{F+t(S)}2 \leq g(S) = g(S')-1=\frac{F+t}2-1 = \frac{F+t-2}2,\end{equation} and consequently, $t(S) \leq t-2$. On the other hand, if $y \in \mathrm{PF}(S') \setminus \{x,F-x\}$, then $y + z \in S' \subset S$, for every $z \in S'$ and, by condition (b), $x+y \in S.$ Thus, $\mathrm{PF}(S') \setminus \{x,F-x\} \subseteq \mathrm{PF}(S)$ and, consequently, $t(S')-2 = t-2 \leq t(S)$. Thus, $t(S) = t-2$. 

From \eqref{ecu2} it follows that $g(S) = \frac{F(S)+t(S)}2$; so, by Proposition \ref{1}, we obtain that $S$ is almost symmetric. 

Notice that in both implications we have obtained that $\mathrm{PF}(S') = \mathrm{PF}(S) \cup \{x,F-x\},$ as immediate consequence. 
\end{proof}

Observe that the set, $\cup_{i=t}^F \mathscr{A}(F,i)$, of almost symmetric numerical semigroups with Frobenius number $F \geq 1$ and type greater than or equal to $t \leq F$ forms a rooted tree with root $M(F)$. Thus, using Theorem \ref{16}, we can formulate an iterative algorithm for computing the whole set $\cup_{i=t}^F \mathscr{A}(F,i)$, where the cases $1 \leq F \leq 4$ can be obtained by direct computation.

Obviously if $t = 1$, our algorithm will compute the whole set $\mathscr{A}(F)$ of almost symmetric numerical semigroups with Frobenius number $F$.

\begin{algorithm}\label{17} 
The following GAP (\cite{GAP}) functions compute the whole set \[\bigcup_{i=t}^F \mathscr{A}(F,i)\] for $1 \leq t \leq F,\ F+t$ even and $F \geq 3$. The first function is nothing but the recursive step, while the second one is the core function that follows from Theorem \ref{16}. The function \texttt{NumericalSemigroupByGaps} from the 
\texttt{NumericalSgps} package (\cite{numericalsgps}) is required.

\medskip
\begin{verbatim}
assemigroupswithfrobeniusnumberuptotype := function(F,t)
  local L,sg,PF,Ga,N;
  L:=[[[F+1 .. 2*F+1],[1 .. F],[1 .. F]]];
  sg:=Union([F-1],[F+1 .. 2*F+1]);
  PF:=Difference([1 .. F],[F-1,1]);
  Ga:=Difference([1 .. F],[F-1]);
  Append(L,[[sg,PF,Ga]]);

  N:=[[sg,PF,Ga]];
  repeat
    N:=Concatenation(List(N, t->CallFuncList(auxiliar,t)));
    Append(L,N);
  until Length(N[1][2]) < t+2;
  return List(L, t->NumericalSemigroupByGaps(t[3]));
end;

auxiliar := function(sg,PF,Ga)
    local L,m,t,i,sg1,PF1,Ga1;
    L := []; m:=sg[1]; t:=Length(PF); F:=PF[t];
    for i in [t-1 .. m-1] do
      sg1:=Union([i],sg);
      PF1:=Difference(PF,[i,F-i]);
      Ga1:=Difference(Ga,[i]);
      if  IsSubset(Ga1,Filtered(Ga1-i,j->(j>0))) 
       and
         (Intersection(PF1+i, Ga1) = [ ]) then
       Append(L,[[sg1,PF1,Ga1]]);
      fi;
    od;
    return L;
end;
\end{verbatim}

A much more refined implementation of our algorithm can be found at the development version site of the \texttt{NumericalSgps}  package (\cite{numericalsgps}): \url{https://github.com/gap-packages/numericalsgps}. 
\end{algorithm}

Notice that algorithm computes also the set of pseudo-Frobenius numbers of every almost symmetric numerical semigroup of Frobenius number $F$ and type up to a fixed $t \geq 1$. Indeed, the key fact of Theorem \ref{16} is that controls the set of pseudo-Frobenius numbers.

\begin{remark}\label{Remark final}
The semigroup algebra $\mathbb{Q}[S] := \bigoplus_{s \in S} \mathbb{Q}\, \chi^s,$ a $S-$graded resolution has the form $$0 \to \bigoplus_{j=1}^{l_{k-1}} A^{r_{k-1\, j}}(-b_{k-1\, j}) \longrightarrow \ldots \longrightarrow \bigoplus_{j=1}^{l_1} A^{r_{1\, j}}(-b_{1\, j}) \longrightarrow A \stackrel{\varphi_0}{\longrightarrow} \mathbb{Q}[S] \to 0,$$ where integer $\beta_i := \sum_{j=1}^{l_j} r_{ij}$ is the rank of the $i-$th syzygy module of $\mathbb{Q}[S]$ and $b_{ij} \neq b_{ij'},\ j \neq j'.$ 

It is well known that $PF(S) = \big\{ b - \sum_{i=1}^k a_i \mid b \in \{b_{k-1\, 1}, \ldots,  b_{k-1\, l_{k-1}} \} \big\}$ (see \cite{MOT}). 

On the other hand, in \cite{OjVi} an algorithm for the computation of the $S-$gra\-ded minimal free resolution of $\mathbb{Q}[S]$ is given. One of the main contribution in that paper is that the resolution can be computed from left to right, that is, starting from the pseudo-Frobenius number of $S$. 

Therefore, our algorithms produce families of numerical semigroups for which $t = \beta_{k-1}$ is a priori known, and furthermore, for which the $S-$gra\-ded minimal free resolution of $\mathbb{Q}[S]$ can be computed starting from left to right.
\end{remark}

\section{Calculation time comparisons}\label{sect time}

In this section, a table with some computational results is presented. We have obtained it running GAP 4.8.6 (\cite{GAP}) in a Intel(R) Core(TM) i5-2450M CPU \@ 2.50GHz. The used implementations of our algorithms are the given above. The function
\begin{center}
 \texttt{AlmostSymmetricNumericalSemigroupsWithFrobeniusNumber} 
\end{center}
is the one included in the version 1.1.8 of the \texttt{NumericalSgps} package (\cite{numericalsgps}). 

For simplicity, we write 
\begin{enumerate}[ (A)]
\item for \texttt{AlmostSymmetricNumericalSemigroupsWithFrobeniusNumber}, 
\item for \texttt{assemigroupswithfrobeniusnumber}, 
\item for \texttt{assemigroupswithfrobeniusnumberuptotype}
\end{enumerate}
The computation time below is expressed in seconds.

\medskip
\begin{center}
\begin{tabular}{||c|c|c|c|c|c|c|c|c|c|c|c|c|c|c|c|c|c|c|c|}
\hline 
Frobenius(S) & 13 & 14 & 15 &  20 & 25 & 30 & 40 \\
\hline\hline 
(A)  & 0.004 & 0.004 & 0.012 & 0.032 & 0.196 & 0.764 & 39.292 \\
\hline 
(B)  & 0 & 0.008 & 0.008 & 0.04 & 0.172 & 0.764 & 38.924 \\
\hline 
(C)  & 0 & 0 & 0.004 & 0.008 & 0.028 & 0.076 & 0.692  \\
\hline 
\end{tabular}
\end{center}

\medskip
It is important to mention that the algorithm given in Section \ref{S3} is better than (C) if one just wants to compute the set $\mathscr{A}(F,t)$ for small $t$. For example, the function \texttt{assemigroupswithfrobeniusnumberandtype} took $0.884$ seconds to compute $\mathscr{A}(50,4)$, whereas a truncated version of (C) took $6.632$ seconds. The reason is clear, the procedure (C) needs to compute first $\mathscr{A}(50,t)$ for every $t$ from $F$ to $4$.

\medskip
The referee proposed us the following algorithm (for simplicity we only show the case $F$ odd): 
\begin{verbatim}
assemigroupswithfrobeniusnumberuptotype2 := function(F,t)
  local F2,t2s,listaAS,s2,L,S; 
  F2:=[1..Int(F/2)];
  t2s:=Combinations(F2,Int(t/2));
  listaAS:=[];;
  for s2 in t2s do 
    L:=Union(s2,F-s2);
    Append(L,[F]);
    S:=NumericalSemigroupsWithPseudoFrobeniusNumbers(L);
    Append(listaAS,S);
  od;
return(listaAS);
end;
\end{verbatim}
This algorithm effectively computes all AS-semigroups with a fixed Frobenius number and a given type. However, the large number of cases that must be verified makes the computation with very slow in comparission to the others. For instance, this algorithm computes the almost symmetric numerical semigroups with Frobenius number equal to 25 in arround 3.5 seconds. Perhaps an optimal choice of \texttt{t2s} and/or a particularization of the algorithm in \cite{DGSRP} to almost symmetric numerical semigroups will produce a faster algorithm.

\begin{remark}\label{rem_fin}
We finish this paper by noting that the computational evidence supports that, for $F$ large enough, the sequence of the cardinalities of the sets of almost symmetric of Frobenius number $F$ and decreasing type are
\[1, 1, 2, 4, 7, 12, 23, 39, 67, 118, 204, 343, 592, 1001,1693, \ldots\]
This happens, for instance, when $F=55$. 
This leads us to formulate the following conjecture:
\begin{enumerate}
\item[$\bullet$] there exists $N$ such that \[\# \mathscr{A}(F,t) = n_{(F-t)/2},\ \text{for every}\ F\ \text{and}\ t \geq N \] i.e. there exists $N$ such that $\# \mathscr{A}(F,t)$ is equal to the number of numerical semigroups with genus $(F-t)/2$, for every $F$ and every $t \geq N$.
\end{enumerate}
Therefore, we could use the $\# \mathscr{A}(F,t)$ to give light to the Bras-Amor\'os' conjecture on the number of numerical semigroups with a given genus (see \cite{Bras}). 

Observe that if our conjecture is true, it will also be an interesting problem to give an upper bound for the largest $N$ in terms of $F$.
\end{remark}

\medskip
\noindent\textbf{Acknowledgement.}
We thank the anonymous referees for their detailed suggestions and comments, which have greatly improved this article. We also thank Pedro A. Garc\'{\i}a-S\'anchez for his invaluable help to refine and improve the implementation of our algorithms.


\end{document}